\documentclass[12pt]{amsart}

\usepackage[utf8]{inputenc}
\usepackage{listings}
\usepackage{graphicx}
\usepackage{pgffor} 
\usepackage{float}
\usepackage{graphics} 
\usepackage{fancyhdr}
\usepackage{amsfonts}
\usepackage{amsmath}
\usepackage{amssymb,amsfonts}
\usepackage{enumerate}
\usepackage{graphicx}
\usepackage{amsmath,amssymb,amsthm}
\usepackage{enumerate}
\usepackage{enumitem}
\usepackage{setspace}
\usepackage{amsthm}
\usepackage{mathrsfs}
\usepackage{nccmath}
\usepackage{mathtools}
\usepackage{appendix}
\usepackage{subcaption}
\usepackage{stmaryrd}
\usepackage{tikz-cd}
\usepackage[OT2,T1]{fontenc}
\usepackage[normalem]{ulem}
\usepackage{mathrsfs}

\graphicspath{{Images/}}

\newcommand{\C}{\mathbb{C}}
\newcommand{\Z}{\mathbb{Z}}

\newcommand{\F}{\mathbb{F}}
\newcommand{\Pn}[1]{\mathbb{P}^{#1}_{k}}

\DeclareFontFamily{U}{wncy}{}
\DeclareFontShape{U}{wncy}{m}{n}{<->wncyr10}{}
\DeclareSymbolFont{mcy}{U}{wncy}{m}{n}
\DeclareMathSymbol{\Sha}{\mathord}{mcy}{"58} 

\newtheorem{thm}{Theorem}[subsection]
\newtheorem*{thm*}{Theorem}
\newtheorem{cor}[thm]{Corollary}
\newtheorem{lem}[thm]{Lemma}
\newtheorem{prop}[thm]{Proposition}

\newtheorem{defn}[thm]{Definition}
\newtheorem*{ack*}{Acknowledgements}
\theoremstyle{definition}
\newtheorem{rem}[thm]{Remark}

\renewcommand{\thethm}{%
  \ifnum\value{subsection}=0
    \thesection.\arabic{thm}%
  \else
    \thesubsection.\arabic{thm}%
  \fi
}

\DeclareSymbolFont{cyrletters}{OT2}{wncyr}{m}{n}
\DeclareMathSymbol{\Sha}{\mathalpha}{cyrletters}{"58}

\title[unramified cohomology over non-closed fields]{On computation of unramified cohomology over non-closed fields}
\author{Wenhao Li}

\address{Courant Institute of Mathematical Sciences, New York University, New York, 10012, U.S.A.}
\email{wl1693@nyu.edu}

\usepackage{filecontents}
\usepackage[style=alphabetic,backend=biber,maxcitenames=50,maxnames=50,maxalphanames=40]{biblatex}
\renewbibmacro{in:}{}
\renewbibmacro*{volume+number+eid}{%
\printfield{volume}%
\setunit{, no.}%
\printfield{number}%
\setunit{\bibeidpunct}%
\printfield{eid}}
\begin{document}

\begin{abstract}
 We give examples of varieties $X$ defined  over a non-algebraically closed field $k$ with nontrivial unramified cohomology, in the case when the field $k$ is of bounded cohomological dimension, or the variety $X$ is a conic bundle over a rational surface and $k$ is an arbitrary field of characteristic different from $2$.  
\end{abstract}

\maketitle

\section{Introduction}

In their famous paper  \cite{AM72}, Artin and Mumford constructed one of the first examples of unirational non rational complex varieties. Building up on this work,  Colliot-Th\'{e}l\`{e}ne and Ojanguren \cite{CTO89} explained this example using  birational invariants given by the unramified cohomology, and introduced a method for computing these invariants for varieties fibred in quadrics over a rational base. Recently, unramified cohomology has been extensively used for detecting the failure of (stable) rationality for many classes of varieties, in the framework of the specialization method developed by Voisin \cite{V15} and Colliot-Th\'{e}l\`{e}ne and Pirutka \cite{CTP16A}, and applied recently in numerous works. For instance, computations of unramified cohomology of quadric bundles over algebraically closed fields appeared in \cite{HkT16}, \cite{HPT18}, \cite{P18},  \cite{ABP18}, \cite{S18}, \cite{S19A}, and \cite{ABBP20}.

In \cite[theorem 11.3.14]{CTS21} and \cite{AP11, P18}, Colliot-Th\'{e}l\`{e}ne and Pirutka constructed examples of conic bundles over a rational surface and, respectively, higher-dimensional quadric fibrations  $\pi:X\to S$ with nontrivial unramified cohomology over $\C$ or $\F_p$. Such fibrations are ramified over specifically constructed hypersurfaces.

While these results primarily focused on either lower degree invariants or algebraically closed base fields, the behavior of higher degree unramified cohomology over non-closed fields remains less explored. In this paper, we generalize Colliot-Th\'{e}l\`{e}ne and Pirutka's configuration to varieties over fields with cohomological dimension $\leq d$. 

Furthermore, we  extend Colliot-Th\'{e}l\`{e}ne's formula \cite[theorem 11.3.14]{CTS21}  for Brauer groups of conic bundles over a surface to non-closed fields with characteristic not $2$.

\medskip

This paper is organized as follows.
Section 2 summarizes the definitions and basic properties of unramified cohomology, relative unramified cohomology, quadric fibrations, and cycle modules. We also recall some local results on residue maps for quadric fibrations.

Section 3 aims at  constructing quadric fibrations with nontrivial unramified cohomology groups over a base field of cohomological dimension $\leq d$. See Proposition \ref{config}. We also give a few explicit examples.

Section 4 extends Colliot-Th\'{e}l\`{e}ne's formula for computing the Brauer group of a conic bundle over a smooth, projective, rational surface to any field with characteristic not $2$. The formula is the following (see theorem \ref{mainformula}):
    \begin{thm}\label{mainformula}
Let $k$ be a field of characteristic $char(k)\neq2$. Let $S$ be a smooth, projective, rational surface over $k$ and let $F$ be the field of functions of $S$. Let $X$ be a threefold equipped with a conic bundle structure $\pi: X\to S$ and let $\alpha\in Br(F)[2]$ be the class corresponding to the quaternion algebra $A$ associated to the conic given by the generic fiber $Q$ of $\pi$. Assume that $\alpha$ is nonzero and that the ramification curve 

$$C\coloneqq \left\{x\in S^{(1)}\mid \partial^2_x(\alpha)\neq0\right\}$$
has smooth irreducible components intersecting transverselly. Let $C=\cup_1^nC_i$ be the irreducible decomposition of $C$, and let 
$$\{(\gamma_i)\}_i\in \oplus_1^n H^1(\kappa(C_i),\Z/2\Z)$$ be the associated family of residues of $\alpha$. For $\zeta=\{(m_i)\}_i\in (\Z/2\Z)^n$, denote

    $$\gamma_{\zeta}=(\gamma_i^{m_i})\in\oplus_{x\in{S}^{(1)}}H^1(\kappa(x),\Z/2\Z).$$ 
    
    Consider the subgroup $H\subset(\Z/2\Z)^n$ of $n$-tuples $(m_i)$ such that 
    \begin{itemize}
    
    \item [(1)] for every $\zeta\in H$, $\gamma_{\zeta}$ lifts\footnote{ $\;$ i.e. there exists a $\alpha_\zeta\in H^2(F,\Z/2\Z)$ such that $\phi_2(\alpha_{\zeta})=\gamma_{\zeta}$ as in definition (\ref{deflifts})} to $H^2$, in particular, if $i\neq j$ and there is a point $P\in C_i\cap C_j$ such that $\partial^1_P(\gamma_i)=\partial^1_P(\gamma_j)\neq0$ then $m_i=m_j$;
        
        \item [(2)] if there is a point $P\in C_i\cap C_j$ such that $\partial^1_P(\gamma_i)=\partial^1_P(\gamma_j)=0$ but $m_i=1\neq m_j$, then $\overline{\gamma_i}=1$ in $\kappa(P)^\times/\kappa(P)^{\times2}.$
    \end{itemize}
    Then $H^2_{nr}(k(X)/k,\Z/2\Z)/H^2(k,\Z/2\Z)=H/(1,\ldots,1)$.
\end{thm}

{\bf Acknowledgments.}
The author was partially supported by NSF grant DMS-2201195. The author is very grateful of Alena Pirutka's thorough guidance throughout this project, and would also like to thank Fedor Bogomolov and Zhijia Zhang for helpful discussions and suggestions.

\section{Reminders and local computations}
\subsection*{Notations}
Let $k$ be a field, and let $n>0$ be an integer invertible in $k$.  Let $\zeta_{n}$ be a primitive $n^{th}$ root of unity.  Denote by $\mu_n$ the \'{e}tale $k$-group scheme of the $n^{th}$ roots of unity over $k$. If $\zeta_n\in k$, one has  isomorphisms 
$\mu_n^{\otimes j}\cong\Z/n\Z
$
for all $j$. Denote by $H^i(k,\mu_n^{\otimes j})$ the Galois cohomology groups.

\subsection{Unramified cohomology}

Let $A$ be a discrete valuation ring, $K$ be the field of fractions of $A$, and $\kappa(v)$ be the residue field. Assume  $(n, char(\kappa(v)))=1$.  One has the residue maps (see e.g. \cite{CT92})
$$
\partial_v^i: H^i(K,\mu_n^{\otimes j})\rightarrow H^{i-1}(\kappa(v),\mu_n^{\otimes(j-1)}). 
$$

\medskip

Given an integral algebraic variety $X$ over $k$, the {\bf unramified cohomology groups} of $X$ are defined as (see {\it loc. cit.})

$$
H^i_{nr}(k(X)/k,\mu_n^{\otimes j})\coloneqq \bigcap_v \operatorname{Ker}[H^i(k(X),\mu_n^{\otimes j})\xrightarrow{\partial^i_v} H^{i-1}(\kappa(v),\mu_n^{\otimes (j-1)})]
$$
where the intersection runs over all discrete valuations $v$ on $k(X)$ of rank $1$, trivial on $k$.

\subsection{Fibrations}

Let $B$ be a smooth projective integral variety over $k$, and let $F=k(B)$ be its function field.
Assume  that $X$ is an integral projective variety, and that one has a dominant map $\pi: X\rightarrow B$  with generic fiber $Q$ a quadric over $F$.

Recall that for a valuation $v$ on $k(X)$ with valuation ring $A\subset k(X)$, the center $M$ of $v$ on $B$ is defined as the image of the closed point of $Spec\ A$ in $B$ under the map $Spec\ A\rightarrow X\rightarrow B$. This induces an injective homomorphism 
$$\varphi:\mathcal{O}_{B,M}\rightarrow A$$ 
such that $\varphi(\mathfrak{m}_{\mathcal{O}_{B,M}})\subset \mathfrak{m}_A$, where $\mathfrak{m}_A$ is the maximal ideal of $A$.
This further induces a field extension 
$\phi:\kappa(M)\hookrightarrow\kappa(v).$ One has the following commutative diagram:
\begin{equation}\label{center}
\begin{tikzcd}
	A^\times && {\kappa(v)^\times} \\
	\\	{\mathcal{O}_{B,M}^\times} && {\kappa(M)^\times}
	\arrow[ from=1-1, to=1-3]
	\arrow["\varphi", hook, from=3-1, to=1-1]
	\arrow[ from=3-1, to=3-3]
	\arrow["\phi"', hook', from=3-3, to=1-3]
\end{tikzcd}
\end{equation}

Given a  point $x\in B$ of positive codimension, let $F_x$ be the field of fractions of the completed local ring $\widehat{\mathcal{O}}_{B,x}$. We also have the notion of relative unramified cohomology $H^i_{nr,\pi}(k(X)/k,\mu_n^{\otimes j})\subset H^i(k(X),\mu_n^{\otimes j})$:

\medskip

\begin{align*}
   & H^i_{nr,\pi}(k(X)/k,\mu_n^{\otimes j})\coloneqq\\ &\operatorname{Im}[H^i(F,\mu_n^{\otimes j})\rightarrow H^i(k(X),\mu_n^{\otimes j})]\bigcap \cap_x\operatorname{Ker}[H^i(k(X),\mu_n^{\otimes j})\rightarrow H^i(F_x(Q),\mu_n^{\otimes j})], 
\end{align*}
where $x$ runs over all scheme points of $B$ of positive codimension (\cite{P23}). 

\bigskip

Note that one has (\cite[Remark 2.2]{P23}):
$$
H^i_{nr,\pi}(k(X)/k,\mu_n^{\otimes j})\subset H^i_{nr}(k(X)/k,\mu_n^{\otimes j}).
$$

We will also use the fact that the following diagram is commutative:
\begin{equation}\label{fac}\begin{tikzcd}
	{H^i(F(Q),\mu_n^{\otimes j})} && {H^i(F_x(Q),\mu_n^{\otimes j})} \\
	\\
	{H^i(F,\mu_n^{\otimes j})} && {H^i(F_x,\mu_n^{\otimes j})}
	\arrow["{\operatorname{res}_{F(Q)/F_x(Q)}}", from=1-1, to=1-3]
	\arrow["{\operatorname{res}_{F/F(Q)}}", from=3-1, to=1-1]
	\arrow["{\operatorname{res}_{F/F_x}}"', from=3-1, to=3-3]
	\arrow["{\operatorname{res}_{F_x/F_x(Q)}}"', from=3-3, to=1-3]
\end{tikzcd}
\end{equation}

\subsection{Fibrations in quadrics}

\begin{thm}(\cite{OVV})\label{ovvker}
    Let $F$ be a field of characteristic $char(F)\neq 2$, and let $q$ be a Pfister neighbor of the Pfister form $\langle \langle a_1,\ldots,a_s\rangle\rangle$ with $a_i\in F^*$. Let $f: Q\rightarrow \operatorname{Spec} F$ be the projective quadric associated to $q$. Then the kernel of the map
    $$
    f^*:H^s(F,\Z/2\Z)\rightarrow H^s(F(Q),\Z/2\Z)
    $$
    is generated by $(a_1,\ldots,a_s)$.
\end{thm}
\begin{proof}
    As the image of  $f^*$ lands in  $H^s_{nr}(F(Q),\Z/2\Z)\subset H^s(F(Q),\Z/2\Z)$, one has
    $$
    \operatorname{Ker}(f^{*})=\operatorname{Ker}[H^s(F,\Z/2\Z)\rightarrow H^s_{nr}(F(Q),\Z/2\Z)].
    $$
    
    By \cite[Lemma 12]{S19A}, a Pfister neighbor $q$ of a Pfister form $q'$ is stably birational to $q$. As a result, if $g:Q'\to \operatorname{Spec} k$ is the projective quadric associated to $q'=\langle \langle a_1,\ldots,a_s\rangle\rangle$ and
    $$
    g^*:H^s(F,\Z/2\Z)\rightarrow H^s(F(Q'),\Z/2\Z)
    $$
    is the induced map, then
    $$
    H^s_{nr}(F(Q),\Z/2\Z)\cong H^s_{nr}(F(Q'),\Z/2\Z),
    $$
    and
    $\operatorname{Ker} f^{*}=\operatorname{Ker} g^{*}.
    $

    By the result of Orlov-Vishik-Voevodsky in \cite{OVV} (see also \cite[Corollary 40.9]{EKM} and its proof),
    $$
    \operatorname{Ker} g^{*}=(a_1,\ldots,a_s).$$
    
    \vspace{-0.7cm}
\end{proof}

Let $k$ be a field of characteristic $char(k)\neq 2$ and let  
\begin{equation}\label{defX}
\pi: X\rightarrow \Pn{n}
\end{equation} be a fibration in quadrics such that its generic fiber $Q$ over $F=k(\Pn{n})$ is a quadric in $2^{d+n-1}+1$ variables which is a Pfister neighbor of the Pfister form $\langle \langle a_1,\ldots,a_d,f_1,\ldots,f_{n-1},g_1g_2\rangle \rangle$, where $a_i\in k^\times\backslash k^{\times2}$, and $f_i,g_1,g_2\in F^{*}$. We will use the following notations for the following three elements:
\begin{align*}
\alpha_{12}&=\left(a_1,\ldots,a_d,f_1,\ldots,f_{n-1},g_1g_2\right)\in H^{d+n}(F, \Z/2\Z),\\
\alpha_1&=\left(a_1,\ldots,a_d,f_1,\ldots,f_{n-1},g_1\right)\in H^{d+n}(F, \Z/2\Z),\\
\alpha_2&=\left(a_1,\ldots,a_d,f_1,\ldots,f_{n-1},g_2\right)\in H^{d+n}(F, \Z/2\Z).
\end{align*}
When it is clear from the context, we will also use the same notations for their images in $H^{d+n}(F(Q),\Z/2\Z)$.

\begin{prop}\label{ker} In the above notations, the kernel of the map $$H^{d+n}(F, \Z/2\Z) \rightarrow H^{d+n}(F(Q), \Z/2\Z)$$ is $\Z/2\Z$, generated by $\alpha_{12}$.
\end{prop}
\begin{proof}
    This follows from Theorem \ref{ovvker}.
\end{proof}

\subsection{Local results}
\

\noindent 
In this section we recall two results from \cite{CTO89}.
\begin{prop}\label{CTO}
    Let $B$ be a discrete valuation ring, with fraction field $K$ and residue field $\kappa(B)$, assume $char(\kappa(B))\neq 2$. Let $s\geq1$ be an integer. Let $\alpha\in H^s_{\acute{e}t}(B,\Z/2\Z)$ and let $\alpha_0\in H^s(\kappa(B),\Z/2\Z)$ be the image of $\alpha$ by the reduction map. Let $b\in K^\times$ be of valuation $m$ in $B$ and let $\beta$ be the class of $b$ in $H^1(K,\Z/2\Z)$. Then
        $$
        \partial_B^{s+1}(\alpha\cup \beta)=m\alpha_0.
        $$
\end{prop}
\begin{proof}
    See \cite[Proposition 1.3]{CTO89}.
\end{proof}
\begin{rem}
    Let $\alpha=(a_1,\ldots,a_s)\in H^s_{\acute{e}t}(B,\Z/2\Z)$ with $a_1,\ldots, a_s$ invertible in $B$, then $\alpha_0=(\overline{a_1},\ldots,\overline{a_s})$, where $\overline{a_i}$ is the image of $a_i$ under the reduction map $B^\times\to\kappa(B)$.
\end{rem}
\begin{rem}
    Similarly, if $\beta\in H^r(K,\Z/2\Z),$ then $\partial_B^{s+r}(\alpha\cup \beta)=\alpha_0\cup\partial_B^r(\beta).$
\end{rem}

\medskip

\begin{prop}\label{nr}
    Let $X$ be as in (\ref{defX}). Suppose the following two conditions are satisfied:
    \begin{itemize}
        \item[(i)] There exist discrete valuation rings $B_1$ and $B_2$ with field of fractions  $F$, such that 
        $$
        \partial_{B_1}^{d+n}(\alpha_1)\neq 0 \quad \mbox{and}\quad \partial_{B_2}^{d+n}(\alpha_2)\neq 0.
        $$

        \item[(ii)] For every discrete valuation ring $B$ with field of fractions $F$,
        $$
        \partial_B^{d+n}(\alpha_1)=0\quad\mbox{or}\quad\partial_B^{d+n}(\alpha_2)=0.
        $$
    \end{itemize}
    Then the image of $\alpha_1$ in $H^{d+n}(F(Q), \Z/2\Z)$ is nonzero, and it is in the unramified cohomology group $H^{d+n}_{nr}(k(X)/k, \Z/2\Z).$
\end{prop}
\begin{proof} This follows from \cite[Proposition 2.1]{CTO89}: we recall the proof for convenience.
    By proposition $\ref{ker}$, $$
    \operatorname{ker}\left[H^{d+n}(F, \Z/2\Z) \rightarrow H^{d+n}(F(Q), \Z/2\Z)\right]=\langle\alpha_{12}\rangle.
    $$
    If $\alpha_1$ is also in the kernel, then we have 
    $$\alpha_1=0\mbox{ in }H^{d+n}(F, \Z/2\Z)\quad\mbox{or}\quad\alpha_2=\alpha_{12}-\alpha_1=0\mbox{ in }H^{d+n}(F, \Z/2\Z).$$ 
    This contradicts condition $(i)$.

    Next, we show that for any discrete valuation ring $A$ on $F(Q)$, 
    $$\partial_A^{d+n}(\alpha_1)=0.
    $$ 
    Indeed, if $A$ contains $F$, then the image of $\alpha_1$ comes from an element of $H^{d+n}_{\acute{e}t}(A,\Z/2\Z)$ and hence has residue $0$. If $A$ does not contain $F$, take 
    $$
    B=A\cap F,
    $$ 
    which is a discrete valuation ring with field of fractions $F$. We have that at least one of the residues $\partial_B^{d+n}(\alpha_1)$ and $\partial_B^{d+n}(\alpha_2)$ vanishes by assumption.
    
    Let $e_{B/A}$ be the valuation of the uniformizer $\pi_B$ of $B$ in $A$. We have the following commutative diagram 
\[\begin{tikzcd}
	{H^{d+n}(F(Q), \mathbb{Z} / 2\Z)} && {H^{d+n-1}(\kappa(A), \mathbb{Z} / 2\Z)} \\
	\\
	{H^{d+n}(F, \mathbb{Z} / 2\Z)} && {H^{d+n-1}(\kappa(B), \mathbb{Z} / 2\Z)}
	\arrow["{\partial_A^{d+n}}", from=1-1, to=1-3]
	\arrow["{\operatorname{res}_{F/F(Q)}}", from=3-1, to=1-1]
	\arrow["{\partial_B^{d+n}}"', from=3-1, to=3-3]
	\arrow["{e_{B/A}\operatorname{res}_{\kappa(B)/\kappa(A)}}"', from=3-3, to=1-3]
\end{tikzcd}\]

         Assume $\partial_B^{d+n}(\alpha_1)=0$, then $
        \partial_A^{d+n}(\alpha_1)=0
        $ by the commutative diagram.
        
        Assume now $\partial_B^{d+n}(\alpha_2)=0$, then   $
        \partial_A^{d+n}(\alpha_2)=0
        $ from the above diagram again.
        Since $\alpha_{12}=0$ in $H^{d+n}(F(Q),\Z/2\Z)$, we have 
        $$
        \partial_A^{d+n}(\alpha_{12})=0
        $$ 
        and 
        $$
            \partial_A^{d+n}(\alpha_1)=\partial_A^{d+n}(\alpha_{12}-\alpha_2)
            =\partial_A^{d+n}(\alpha_{12})-\partial_A^{d+n}(\alpha_2)
        =0.
        $$
\end{proof}

\section{Case of bounded cohomological dimension}
In this section, we first give two auxiliary results for computing residues of the symbol $(a_1,\ldots,a_s)$ when $k$ is a field of cohomological dimension $\leq d$. Then, we give a configuration for constructing examples of quadric fibrations with nontrivial unramified cohomology groups $H^{d+n}_{nr}(k(X)/k, \Z/2\Z)$. Finally, we provide some explicit examples constructed this way.
\subsection{Computing residues}

\begin{prop}\label{cod} Let $k$ be a field of cohomological dimension $\leq d$, and assume $char (k)\neq2$. 
    Let $B$ be a discrete valuation ring with field of fractions $k(\Pn{n})=F$, centered at a point $M\in\Pn{n}$ of codimension $c$. Let 
     $$(a_1,\ldots,a_s)\in H^s(F,\Z/2\Z).$$
    Assume that:
    \begin{itemize}
    \item[(i)] $s\geq d+n-c+1$;
    \item [(ii)]  up to multiplication by a square, at least $d+n-c+1$ of the entries $\{a_i\}$ are invertible in $\mathcal{O}_{\Pn{n},M}.$  
    \end{itemize}
    Then $$\partial^s_B(a_1,\ldots,a_s)=0.$$
\end{prop}
\begin{proof}
    If $s=d+n-c+1$, then by condition {\it(ii)} every entry is invertible up to a square in $\mathcal{O}_{\Pn{n},M}$, and the claim is true since $(a_1,\ldots,a_s)\in H^s_{\acute{e}t}(B,\Z/2\Z)$. Assume $s\geq d+n-c+2$. Without loss of generality, assume that the first $d+n-c+1$ entries are invertible. We use proposition \ref{CTO} (and remarks after): 
    \begin{multline*}
            \partial_B^s(a_1,\ldots,a_s)=\\(\overline{a_1},\ldots,\overline{a}_{d+n-c+1})\cup \partial_B^{s-(d+n-c+1)}(a_{d+n-c+2},\ldots,a_{s})
    \in H^{s-1}(\kappa(B),\Z/2\Z).
    \end{multline*}
    
    By assumption, each $a_i$ with $i\leq d+n-c+1$ is invertible in $\mathcal{O}_{\Pn{n},M}$, so its image in $\kappa(B)$ comes from $\kappa(M)$ by commutative diagram (\ref{center}). Hence the cup product 
    $$
    (\overline{a_1},\ldots,\overline{a}_{d+n-c+1})\in H^{d+n-c+1}(\kappa(B),\Z/2\Z)
    $$ 
    lies in the image of the restriction map
    $$
    H^{d+n-c+1}(\kappa(M),\Z/2\Z)\rightarrow H^{d+n-c+1}(\kappa(B),\Z/2\Z).
    $$
    
    However, since $M\in\Pn{n}$ is of codimension $c$, the field $\kappa(M)$ is of cohomological dimension $d+n-c<d+n-c+1$. So 
    $$
    H^{d+n-c+1}(\kappa(M),\Z/2\Z)=0,
    $$ 
    and hence 
    $
    \partial_B^s(a_1,\ldots,a_s)=0.
    $
\end{proof}

\medskip

\begin{prop}\label{sq}
Let $k$ be a field with $char(k)\neq2$.
    Let $B$ be a discrete valuation ring with field of fractions $k(\Pn{n})=F$, centered at a point $M\in\Pn{n}$ of codimension $c$. Let 
    $
    (a_1,\ldots,a_s)\in H^s(F,\Z/2\Z).
    $ 
    Assume that at least one of $a_i$ is in $\mathcal{O}_{\Pn{n},M}^\times$, and its class $\overline{a_i}$ in $\kappa(M)$ is a square. Then 
    $$
    \partial_B^s(a_1,\ldots,a_s)=0.
    $$
\end{prop}
\begin{proof}
    By Hensel lemma, if $\overline{a_i}$ is a square in $\kappa(M)$, then $a_i$ is also a square in the completion $\widehat{\mathcal{O}}_{\Pn{n},M}$. Let $\widehat B$ be the completion of $B$ and let $\widehat F$ be the completed valued field; note that the residue field of $\widehat{B}$ is $\kappa(B)$. As in (\ref{center}), one has $\widehat{\mathcal{O}}_{\Pn{n},M}\subset \widehat{B}$. Since the residue map factors through the completion $H^s(\widehat F,\Z/2\Z)$, and the image of $a_i$ in $H^1(\widehat F,\Z/2\Z)$ is zero,  proposition \ref{CTO} (applied to $\widehat B$) concludes the proof.
\end{proof}

\bigskip

\subsection{Examples}$\;$\\

In this section we generalize a configuration of Colliot-Th\'{e}l\`{e}ne and Ojanguren in \cite[Exemple 3.3]{CTO89}. 

\medskip

\textbf{General Configuration.} Let $k$ be a field of cohomological dimension $d$ with $char (k)\neq2$. Let $a_1,\ldots,a_d\in k^\times$ such that 
$$
(a_1,\ldots,a_d)\in H^d(k,\Z/2\Z)
$$
is nontrivial. 

Let $h_1,\ldots,h_{2n-2},\gamma_1,\gamma_2,\gamma_3,\gamma_4$ be $2n+2$ distinct homogeneous forms of degree $r$ in $\Pn{n}$. For $i=1,2,3,4$, denote by $\Lambda_{i}$ the four sets of homogeneous forms
$$\Lambda_{i}=\left\{\gamma_i+\delta_1h_1+\ldots+\delta_{2n-2}h_{2n-2}\,\vert \, \delta_j\in\{0,1\}\right\}.$$

Consider also the set of $n+1$ tuples 
\begin{multline*}
    E=\big\{(e_1,\ldots,e_{n-1},\varepsilon_1,\varepsilon_2)\,\vert \\ 
    e_i\in\{h_{2i-1},h_{2i}\},\varepsilon_j\in\Lambda_{2j-1}\cup\Lambda_{2j}, e_i\neq e_j\text{ and }\epsilon_i\neq\epsilon_j\text{ for }i\neq j\big\}.
    \end{multline*}
We make the following assumptions: 
\begin{itemize}
    \item[(A1)] Each set $\Lambda_i$ has cardinality $2^{2n-2}$;
    \item[(A2)] For every $(e_1,\ldots,e_{n-1},\varepsilon_1,\varepsilon_2)\in E$, we have that the hypersurfaces defined by $e_1,\ldots, e_{n-1},\varepsilon_1,\varepsilon_2$ intersect transversely: at any point $M$ lying in the intersection of $m$ of these hypersurfaces, the local defining equations extend to a regular system of parameters for $\mathcal{O}_{\Pn{n},M}$.
    \item[(A3)] The intersection of the hypersurfaces defined by $h_1,h_3,\ldots,h_{2n-3},\gamma_1$ and the intersection of the hypersurfaces defined by $h_1,h_3,\ldots,h_{2n-3},\gamma_3$ each contains a point $P$ such that the degree of its residue field extension $[\kappa(P) : k]$ is odd.
\end{itemize}

\bigskip

\begin{thm}
\label{config}
    With the above notations, let 
    $$g_1=\frac{\prod_{\lambda_i\in\Lambda_1}\lambda_i}{\prod_{\lambda_i\in\Lambda_2}\lambda_i}, \
g_2=\frac{\prod_{\lambda_i\in\Lambda_3}\lambda_i}{\prod_{\lambda_i\in\Lambda_4}\lambda_i},\
    f_i=\frac{h_{2i-1}}{h_{2i}},\  i=1,\ldots, n-1 .$$
    Let $\pi:X\rightarrow \Pn{n}$ be a fibration in quadrics such that the generic fiber $Q$ is a quadric in $2^{d+n-1}+1$ variables which is a  Pfister neighbor of the Pfister quadric $\langle \langle a_1,\ldots,a_d,f_1,\ldots,f_{n-1},g_1g_2\rangle \rangle$. Then the image of the element
    $$\alpha_1=\left(a_1,\ldots,a_d,f_1,\ldots,f_{n-1},g_1\right)\in H^{d+n}(F, \Z/2\Z)$$ in $H^{d+n}(F(Q), \Z/2\Z)
    $
    is nontrivial, and it lies in the relative unramified cohomology group
    $ H^{d+n}_{nr,\pi}(k(X)/k, \Z/2\Z).
    $
    
\end{thm}
\begin{proof}
    We first verify that $\alpha_1$ is nontrivial using Proposition \ref{nr} ({\it{i\,}}).

    Let $B_1$ be the discrete valuation ring with field of fractions $F$ centered at (the generic point of) $\gamma_1\in\Lambda_1$. Then by proposition \ref{CTO}, 
    $$
    \partial_{B_1}^{d+n}(\alpha_1)=(a_1,\ldots,a_d,f_1,\ldots,f_{n-1}).
    $$ 
    
    We claim that $ \partial_{B_1}^{d+n}(\alpha_1)$ is nonzero. We prove this by successively taking the residues of this element at (irreducible components of successive intersections of) $h_1=0 , h_3=0,\ldots,h_{2n-3}=0$. By (A2), these hypersurfaces  intersect transversely, so that at each step  the uniformizing parameter is given by (the image of) $h_i$,  hence, using proposition \ref{CTO}, at the end we obtain the class  
    $$
    (a_1,\ldots,a_d)\in H^d(\kappa(P),\Z/2\Z)
    $$ 
    where $P$ is a point in the intersection of $h_1,h_3,\ldots,h_{2n-3},\gamma_1$ such that  $[\kappa(P) : k]$ is odd. The existence of $P$ is guaranteed by (A3). This class in $H^d(\kappa(P),\Z/2\Z)$ is the image under the restriction map of the symbol $
    (a_1,\ldots,a_d)\in H^d(k,\Z/2\Z)
    $. Composing with the corestriction map, we have a map that factors through $H^d(\kappa(P),\Z/2\Z)$:
    $$ \operatorname{cor}_{\kappa(P)/k} \circ \operatorname{res}_{\kappa(P)/k} : H^d(k, \mathbb{Z}/2\mathbb{Z}) \to H^d(k, \mathbb{Z}/2\mathbb{Z})
    $$
    sending $
    (a_1,\ldots,a_d)
    $ to $[\kappa(P) : k](a_1,\ldots,a_d)=(a_1,\ldots,a_d)$.
    Since this element was chosen to be nontrivial, it follows that
    $$
    \partial_{B_1}^{d+n}(\alpha_1)\neq0.
    $$
    
    We can choose $B_2$ to be centered at $\gamma_3$, and the same argument applies. By Proposition \ref{nr} ({\it{i\,}}), $\alpha_1$ is nontrivial in $H^2(F(Q),\Z/2\Z)$.
    
    Next let $x\in\Pn{n}$ be a point of positive codimension $c$. We need to prove that 
    $$\alpha_1\in\operatorname{Ker}[H^{d+n}(k(X),\Z/2\Z)\rightarrow H^{d+n}(F_x(Q),\Z/2\Z)],$$ 
    where $F_x$ is the field of fractions of the completed local ring $\widehat{\mathcal{O}}_{\Pn{n},x}$. Note that $H^s_{\acute{e}t}(\widehat{\mathcal{O}}_{\Pn{n},x},\Z/2\Z)=H^s(\kappa(x),\Z/2\Z)=0$ for $s\geq d+n-c+1$ by cohomological dimension. Also recall that we have the factorization
    \begin{equation}\label{cohofac}
        H^s(F,\Z/2\Z)\rightarrow H^s(F_x,\Z/2\Z)\rightarrow H^s(F_x(Q),\Z/2\Z).
    \end{equation}
    Hence it suffices to prove $\alpha_1\in\operatorname{Ker}[H^{d+n}(k(X),\Z/2\Z)\rightarrow H^{d+n}(F_x,\Z/2\Z)]$.
    
    Define
    $$
    D=\operatorname{supp}(f_1)\cup\ldots \cup\operatorname{supp}(f_{n-1})\cup\operatorname{supp}(g_1)\cup\operatorname{supp}(g_2),
    $$
    and assume $x$ lies on exactly $c'$ irreducible components in $D$. By assumption (A2), $c'\leq c$.

    If $c'=0$, then every entry in $\alpha_1$ is invertible in $\widehat{\mathcal{O}}_{\Pn{n},M}$, and the image of $\alpha_1$ in $H^{d+n}(F_{x},\Z/2\Z)$ comes from $H^{d+n}_{\acute{e}t}(\widehat{\mathcal{O}}_{\Pn{n},x},\Z/2\Z)$. But $H^{d+n}_{\acute{e}t}(\widehat{\mathcal{O}}_{\Pn{n},x},\Z/2\Z)=0$ because $x$ is of positive codimension and $k$ is of cohomological dimension $d$.

    Now assume $c'>0$. If $x\in \operatorname{supp}(g_1)$, then at most $c^\prime-1$ irreducible components in $D\setminus \operatorname{supp}(g_1)$ contain $x$. Hence at least $d+n-(c'-1)=d+n-c'+1\geq d+n-c+1$ entries of $\alpha_2$ are invertible in $\widehat{\mathcal{O}}_{\Pn{n},M}$. Then the image of $\alpha_2$ in $H^{d+n}(F_{x},\Z/2\Z)$ can be written as a cup product $\alpha_{2}^{inv}\cup\alpha_{2}^\prime$, where $\alpha_{2}^{inv}$ comes from $H^{d+n-c+1}_{\acute{e}t}(\widehat{\mathcal{O}}_{\Pn{n},x},\Z/2\Z)$. But $H^{d+n-c+1}_{\acute{e}t}(\widehat{\mathcal{O}}_{\Pn{n},x},\Z/2\Z)=0$, so $\alpha_2$ is $0$ in $H^{d+n}(F_{x},\Z/2\Z)$. By the factorization (\ref{cohofac}), 
    $$\alpha_2\in\operatorname{Ker}[H^{d+n}(k(X),\Z/2\Z)\rightarrow H^{d+n}(F_x(Q),\Z/2\Z)].$$
    Applying proposition \ref{ker} in the local setting, we see that 
    $$\alpha_1=\alpha_{12}-\alpha_2\in\operatorname{Ker}[H^{d+n}(F_x,\Z/2\Z)\rightarrow H^{d+n}(F_x(Q),\Z/2\Z)].$$

    If $x\in\operatorname{supp}(g_2)$, then by the same argument we have that at least $d+n-c+1$ entries of $\alpha_1$ are invertible, and that
    $$
    \alpha_1\in\operatorname{Ker}[H^{d+n}(k(X),\Z/2\Z)\rightarrow H^{d+n}(F_x(Q),\Z/2\Z)].$$

If $x\not\in \operatorname{supp}(g_1)\cup \operatorname{supp}(g_2)$, then there exists at least one $f_i$ such that $x\in\operatorname{supp}(f_i)$. Hence $g_1$ reduces to a nonzero square in $\kappa(x)$ by construction. By Hensel lemma, $g_1$ is a square in $F_x$. Hence $\alpha_1$ is $0$ in $H^{d+n}(F_x,\Z/2\Z)$.

\end{proof}

\begin{rem}
    Since the relative unramified cohomology groups are subgroups of the unramified cohomology groups, $\alpha_1$ is also a nontrivial unramified element.
\end{rem}
We can  also simplify the denominators of $g_i$.
\begin{thm}
    Let 
    $$
    g_1=\frac{\prod_{\lambda_i\in\Lambda_1}\lambda_i}{h_{j_1}^{2^{2n-2}}},\ g_2=\frac{\prod_{\lambda_i\in\Lambda_3}\lambda_i}{h_{j_2}^{2^{2n-2}}},\ f_i=\frac{h_{2i-1}}{h_{2i}},\  i=1,\ldots, n-1
    $$ 
    where 
    $$
    h_{j_1},h_{j_2}\in\{h_1,\ldots,h_{2n-2}\}.
    $$
    Let $\pi:X\rightarrow \Pn{n}$ be a fibration in quadrics such that the generic fiber $Q$ is a quadric in $2^{d+n-1}+1$ variables which is a  Pfister neighbor of the Pfister quadric $\langle \langle a_1,\ldots,a_d,f_1,\ldots,f_{n-1},g_1g_2\rangle \rangle$. Then the element
    $$
    \alpha_1=\left(a_1,\ldots,a_d,f_1,\ldots,f_{n-1},g_1\right)\in H^{d+n}_{nr}(k(X)/k, \Z/2\Z)$$ 
    is nontrivial. Furthermore,
    $$
    \alpha_1\in H^{d+n}_{nr,\pi}(k(X)/k, \Z/2\Z).
    $$
\end{thm}
\begin{proof}
    Since the cohomology groups have coefficients $\Z/2\Z$, multiplying any entry of $\alpha_1$ by a nonzero  square does not change the cohomology class. Pick $h_{t_1},h_{t_2}\not \in \{h_1,\ldots,h_{2n-2}\}\cup\Lambda_1\cup\Lambda_3$ two homogeneous forms of degree $r$ such that for every tuple $(e_1,\ldots,e_{n-1},\varepsilon_1,\varepsilon_2)$ where $e_i\in\{h_{2i-1},h_{2i}\}$ and $\epsilon_j\in \Lambda_{2j-1}\cup h_{t_j}$, we have that the hypersurfaces defined by $e_1,\ldots, e_{n-1},\varepsilon_1,\varepsilon_2$ intersect transversely.
    Replace $g_i$ ($i=1,2$) by $g_i\frac{h_{j_i}^{2^{2n-2}}}{h_{t_i}^{2^{2n-2}}}$. Then the proof of Proposition \ref{config} carry over completely.
\end{proof}
\textbf{Example 1.} Let $k=\C(z_1,z_2)$, so $d=2$. Let $n=2$, and denote by $x,y,u$ the coordinates of  $\Pn{2}$. Let $a_1,a_2$ be non-squares in $k^\times$ such that $(a_1,a_2)\in H^2(k,\Z/2\Z)$ is nontrivial. 
Let
$$
\begin{aligned}[c]
    h_1&=x\\
    \gamma_1&=x+2y+4u
\end{aligned}\quad
\begin{aligned}[c]
    h_2&=y\\
    \gamma_3&=x+3y+5u.
\end{aligned}
$$
These choices satisfy conditions (A1), (A2) and (A3). Explicitly we have
\begin{align*}
      f&=\frac{x}{y},\\
    g_1&=\frac{(2x+2y+4u)(x+3y+4u)(2x+3y+4u)(x+2y+4u)}{x^4},\\
    g_2&=\frac{(2x+3y+5u)(x+4y+5u)(2x+4y+5u)(x+3y+5u)}{y^4}.
\end{align*}
    
    Let $X\rightarrow \Pn{2}$ be a fibration in quadrics such that the generic fiber $Q$ is a quadric in $9$ variables and is a  Pfister neighbor of $\langle \langle a_1,a_2,f,g_1g_2\rangle \rangle$, then $(a_1,a_2,f,g_1)$ is a nontrivial element in $H^{4}_{nr,\pi}(k(X)/k,\Z/2\Z)$.

\medskip
\textbf{Example 2.} Let $C_1,C_2,C_3$ be three distinct smooth plane curves of degree $r$ that do not have a common point. Assume $C_1$ and $C_3$ are chosen such that their intersection contains at least one closed point of odd degree over $k$ (for instance, a $k$-rational point). Denote by $c_i$ the defining equations of $C_i$. Consider the set 
$$
\Sigma_3=\left\{c_1=0, c_3=0\right\}\cup\left\{c_1=0, c_3+c_2=0\right\}\cup\left\{c_2=0, c_3=0\right\}\cup\left\{c_2=0, c_3+c_1=0\right\}
$$
of intersection points. Pick $C_4$ to be a smooth curve of degree $r$ that avoids $\Sigma_3$. Similarly, we define the set of intersections
$$
\Sigma_4=\left\{c_1=0, c_4=0\right\}\cup\left\{c_1=0, c_4+c_2=0\right\}\cup\left\{c_2=0, c_4=0\right\}\cup\left\{c_2=0, c_4+c_1=0\right\}
$$
and choose $C_5$ to be a smooth curve with degree $r$ that avoids $\Sigma_3\cup\Sigma_4$ and that its intersection with $C_1$ contains at least one closed point of odd degree over $k$. We define $\Sigma_5$ in the same way:
$$
\Sigma_5=\left\{c_1=0, c_5=0\right\}\cup\left\{c_1=0, c_5+c_2=0\right\}\cup\left\{c_2=0, c_5=0\right\}\cup\left\{c_2=0, c_5+c_1=0\right\}
$$
and pick $C_6$ to be a smooth curve of degree $r$ that avoids $\Sigma_3\cup\Sigma_4\cup\Sigma_5$. Let
\begin{align*}
    f&=\frac{c_1}{c_2},\\
    g_1&=\frac{c_3(c_3+c_1)(c_3+c_2)(c_3+c_1+c_2)}{c_4(c_4+c_1)(c_4+c_2)(c_4+c_1+c_2)},\\
    g_2&=\frac{c_5(c_5+c_1)(c_5+c_2)(c_5+c_1+c_2)}{c_6(c_6+c_1)(c_6+c_2)(c_6+c_1+c_2)}.
\end{align*}

Let $X\rightarrow \Pn{2}$ be a fibration in quadrics such that the generic fiber $Q$ is a quadric in $9$ variables and is a  Pfister neighbor of $\langle \langle a_1,a_2,f,g_1g_2\rangle \rangle$, then $(a_1,a_2,f,g_1)$ is a nontrivial element in $H^{4}_{nr,\pi}(k(X)/k,\Z/2\Z)$.

\begin{rem}
    The intersection of $C_1$ and $C_3$ (or $C_5$) is a dimension $0$ subscheme of degree $r^2$. The degree of this subscheme is equal to the sum of degrees of its closed points:
    $$
    \sum_{P_i\ \text{closed points}} [\kappa(P_i) : k] = r^2.
    $$
    
    If $r$ is odd, then we are always guranteed a closed point of odd degree. Then at each step we are taking away a finite set of points, so for almost all six tuples $(C_1,C_2,C_3,C_4,C_5,C_6)$ of smooth curves with degree $r$ this is a valid example.

    If $r$ is even, however, (A3) becomes a nontrivial condition, and we no longer have that for almost all six tuples the example works.
\end{rem}

\medskip
\textbf{Example 3.} Let $k=\C(z)$, so $d=1$. Let $n=3$, and denote by $x,y,u,v$ the coordinates of $\Pn{3}$. Let $a_1$ be a non-square in $k^\times$ and
$$\begin{aligned}[c]
    h_1&=x\\
    h_3&=u\\
    \gamma_1&=x + 2y + 4u + 6v
\end{aligned}\quad
\begin{aligned}[c]
    h_2&=y\\
    h_4&=v\\
    \gamma_3&=x + 7y + 21u + 29v.
\end{aligned}$$
These choices satisfy conditions (A1), (A2), and (A3). Explicitly we have
\begin{align*}
      f_1&=\frac{x}{y},\\
      f_2&=\frac{u}{v},\\
    g_1&=\frac{\prod_{\delta_x,\delta_y,\delta_u,\delta_v\in\{0,1\}}(x + 2y + 4u + 6v+\delta_x x+\delta_y y+\delta_u u+\delta_v v)}{u^{16}},\\
    g_2&=\frac{\prod_{\delta_x,\delta_y,\delta_u,\delta_v\in\{0,1\}}(x + 7y + 21u + 29v+\delta_x x+\delta_y y+\delta_u u+\delta_v v)}{y^{16}}.\\
\end{align*}
    
    Let $X\rightarrow \Pn{3}$ be a fibration in quadrics such that the generic fiber $Q$ is a quadric in $9$ variables and is a  Pfister neighbor of $\langle \langle a_1,f_1,f_{2},g_1g_2\rangle \rangle$, then $(a_1,f_1,f_2,g_1)$ is a nontrivial element in $H^{4}_{nr,\pi}(k(X)/k,\Z/2\Z)$.

\bigskip

\section{Conic bundles over $\Pn{2}$}

Let $k$ be a field of characteristic $char(k)\neq2$. Let $S$ be a smooth, projective, rational surface over $k$ and let $F$ be the field of functions of $S$. Let $X$ be a threefold equipped with a conic bundle structure 
\begin{equation}\label{fibration}
    \pi: X\to S,
\end{equation} 
and let $\alpha\in Br(F)[2]$ be the class corresponding to the quaternion algebra associated to the conic given by the generic fiber. Assume that $\alpha$ is nonzero and that the ramification curve 
$$C\coloneqq \left\{x\in S^{(1)}\mid \partial^2_x(\alpha)\neq0\right\}$$
has smooth irreducible components intersecting transverselly. Let $C=\cup_1^nC_i$ be the irreducible decomposition of $C$, and let 
$$\{(\gamma_i)\}_i\in \oplus_1^n H^1(\kappa(C_i),\Z/2\Z)$$ be the associated family of residues of $\alpha$. Recall the Gersten complex:
    \begin{multline*}
        0\to H^2(k,\Z/2\Z)\to H^2(F,\Z/2\Z)\to\\
        \stackrel{\phi_2}{\to}\oplus_{x\in{S}^{(1)}}H^1(\kappa(x),\Z/2\Z)\stackrel{\phi_1}\to\oplus_{P\in S^{(2)}} H^0(\kappa(P),\Z/2\Z),
    \end{multline*}
     where $\phi_2=\oplus_{x\in S^{(1)}}\partial^2_x$ and $\phi_1=\oplus_{P\in S^{(2)}}\partial^1_P$.

\bigskip
    
Let $\zeta=\{(m_i)\}_i\in (\Z/2\Z)^n$, and assume     
    $$\gamma_{\zeta}=(\gamma_i^{m_i})\in\operatorname{Ker}[\oplus_{x\in{S}^{(1)}}H^1(\kappa(x),\Z/2\Z)\to\oplus_{P\in S^{(2)}} H^0(\kappa(P),\Z/2\Z)].$$

\medskip

\begin{defn}
    We say that $\gamma_{\zeta}$ {\bf lifts to $H^2$} if there exists  $\alpha_\zeta\in H^2(F,\Z/2\Z)$ such that 
\begin{equation}\label{deflifts}
\phi_2(\alpha_{\zeta})=\gamma_{\zeta}.
\end{equation}
If this is the case, let 
\begin{equation}\label{betadef}
\beta_\zeta\mbox{ be the image of }\alpha_\zeta
\end{equation}
under the map $H^2(F,\Z/2\Z)\to H^2(k(X),\Z/2\Z)$.
\end{defn}
\begin{rem}
    The Gersten complex is exact for $\mathbb{A}^2$ (\cite{MR96}, Proposition 8.6). However, it fails to be exact at degree $1$ for $\mathbb{P}^2$ over a field $k$ with cohomological dimension greater than $0$. To see this, let $c\in H^1(k,\Z/2\Z)$ represents a nontrivial constant class. Consider the element $\gamma\in \oplus_{x\in{S}^{(1)}}H^1(\kappa(x),\Z/2\Z)$ with $c$ at the generic point of the line defined by $x=0$ and $0$ elsewhere. Then $\gamma$  does not lift to a global class in $H^2(F,\Z/2\Z)$.

    More generally, one has the following exact sequence: 
     \begin{equation}
        0\to H^1(S,\mathcal{K}_2)/2\to H^{1}(S,\mathcal{H}^2)\to CH^2(S)[2]\to 0, 
    \end{equation}
    where the middle term can be identified with the first homology group of the Gersten complex at $\oplus_{x\in{S}^{(1)}}H^1(\kappa(x),\Z/2\Z)$
    (see \cite{CT93}, \cite{DQ73}, and \cite{BO74} for the details and the definition of $\mathcal{K}_2$ and $\mathcal{H}^2$).
    Since $S$ is rational over $k$, its Chow groups of zero cycles $CH^2(S)\cong CH_0(S)\cong \Z$ is torsion free. We therefore have the isomorphism
    $$
H^1(S,\mathcal{K}_2)/2\cong H^{1}(S,\mathcal{H}^2).
    $$
    Denote by $\bar{k}$ the algebraic closure of $k$, $\bar{S}=S\times_k \bar{k}$, and $G=Gal(\bar{k}/k)$ the Galois group. In \cite{CTR85}, Colliot-Th\'{e}l\`{e}ne and Raskind established that 
    $$
    H^1(S,\mathcal{K}_2)\cong (H^1(\bar{S},\mathcal{K}_2))^G,
    $$
    and that (also see \cite{AP11})  $\operatorname{Pic}\bar{S}\otimes \bar{k}^*\cong H^1(\bar{S},\mathcal{K}_2).
    $
    Hence the failure of exactness of the sequence is controlled by the group $(\operatorname{Pic}\bar{S}\otimes \bar{k}^*)^G/2$.

\end{rem}

The goal of the main theorem of this section is to provide a formula for the unramified cohomology group $H^2_{nr}(k(X)/k,\Z/2\Z)$. The proof relies on the computation of $\partial^2_v(\beta_\zeta)$ with respect to different valuations $v$ on $k(X)$. We will stratify the local computations based on the codimension and the geometry of the center of $v$.

\subsection{Local forms}

We first recall:

\begin{thm}\label{pur}
    Let $i > 0$, $n > 0$ and $j$ be integers. Let $k$ be a field of characteristic prime to $n$. Let $A$ be a semi-local ring of a smooth integral $k$-variety $Y$, with fraction field $F = k(Y)$, and let $\alpha$ be an element of $H^i(F,\mu_n^{\otimes j} )$. If for each height one prime $p$ of $A$, $\alpha$ belongs to the image of $H^i_{\acute{e}t}(A_p,\mu_n^{\otimes j})$, then $\alpha$ comes from a (unique) element of $H^i_{\acute{e}t}(A,\mu_n^{\otimes j})$.
\end{thm}
\begin{proof}
    See \cite[Theorem 3.8.2]{CT92}
\end{proof}

Next we will need results of Saltman \cite{S07} for local forms of symbols on surfaces: we will apply these results for $\alpha_0=\alpha$ or $\alpha_{\zeta}$.

\begin{thm}\label{S2.1}
    Let $k$ be a field of characteristic $char(k)\neq2$ and let $S$ be a smooth and projective surface over $k$. Let $\alpha_0\in H^2(F,\Z/2\Z)$ be an element such that the ramification locus of $\alpha_0$ on $S$ is a simple normal crossings divisor.
    \begin{itemize}
        \item [a)] If $M$ is a point only on one ramification locus component $C_i$, then 
        $$\alpha_0=\alpha_0^\prime+(u,\pi_i)$$ 
        in $H^2(F,\Z/2\Z)$, where $\alpha_0^\prime$ is an element in $H_{\acute{e}t}^2(\mathcal{O}_{S,M},\Z/2\Z)$, $u$ is a unit in $\mathcal{O}_{S,M}$, and $\pi_i$ is a local parameter of $C_i$ at $M$.
        \item [b)] Suppose $M$ is a nodal point contained in $C_i$ and $C_j$. Let $\pi_i$ and $\pi_j$ be  local parameters of $C_i$ and $C_j$ at $M$. Then one of the following holds:
        \begin{itemize}
            \item [i)] $\alpha_0=\alpha_0^\prime+(u,\pi_i)+(v,\pi_j),$ and $\partial^2_{C_i}(\alpha_0)$ is unramified at $M$;
            \item [ii)] $\alpha_0=\alpha_0'+(u\pi_j,v\pi_i),$ and $\partial^2_{C_i}(\alpha_0)$ is ramified at $M$,
        \end{itemize}
         where $\alpha_0^\prime$ is an element in $H_{\acute{e}t}^2(\mathcal{O}_{S,M},\Z/2\Z)$, and $u, v$ are units in $\mathcal{O}_{S,M}$.
    \end{itemize}
\end{thm}
\begin{proof}
    See \cite[Theorem 2.1]{S07}.
\end{proof}

\begin{cor}\label{S2.5}
    Assume we are in the case b) i) of Theorem \ref{S2.1}. Denote by $\overline{\gamma_i}$ and $\overline{\gamma_j}$ the reductions of $\gamma_i=\partial^2_{C_i}(\alpha_0)$ and $\gamma_j=\partial^2_{C_j}(\alpha_0)$ in $H^1(\kappa(M),\Z/2\Z)$. If $\overline{\gamma_i}\neq \overline{\gamma_j}$, then $\alpha_0$ has index larger than $2$ over $K_M$, the fraction field of the completion of the local ring at $M$.
\end{cor}
\begin{proof}
    See \cite[Theorem 2.5]{S07}.
\end{proof}

Next, we compute residues with the setup in (\ref{fibration}).
\begin{lem}\label{cod1}
    Let $X$ be as in (\ref{fibration}), and let $\beta_\zeta$ be as in (\ref{betadef}). Let $v$ be a valuation on $k(X)$ centered at a codimension $1$ point $M\in S$. Then $\partial^2_v(\beta_\zeta)=0$.
\end{lem}
\begin{proof}
    We have two cases:
        \begin{itemize}
            \item If $M\neq C_i$ for any $C_i$ or $M=C_i$ but $m_i=0$, then $\partial^2_M(\alpha_\zeta)=0$, which implies $\partial^2_v(\beta_\zeta)=0$.

            \item If $M=C_i$ with $m_i=1$, then we use the following commutative diagram
            \[\begin{tikzcd}
	{H^{2}(k(X), \mathbb{Z} / 2\Z)} && {H^{1}(\kappa(M), \mathbb{Z} / 2\Z)} \\
	\\
	{H^{2}(F, \mathbb{Z} / 2\Z)} && {H^{1}(\kappa(C_i), \mathbb{Z} / 2\Z)}
	\arrow["{\partial^2_v}", from=1-1, to=1-3]
	\arrow["{\operatorname{res}_{F/k(X)}}", from=3-1, to=1-1]
	\arrow["{\partial^2_i}"', from=3-1, to=3-3]
	\arrow["{e_{C_i/M}\operatorname{res}_{\kappa(C_i)/\kappa(M)}}"', from=3-3, to=1-3]
\end{tikzcd}\]
        to see that \begin{align*}
            \partial^2_v(\beta_\zeta)&=\partial^2_v\circ\operatorname{res}_{F/k(X)}(\alpha_\zeta-\alpha)\\
            &=e_{C_i/M}\operatorname{res}_{\kappa(C_i)/\kappa(M)}\circ \partial^2_i(\alpha_\zeta-\alpha)\\
            &=e_{C_i/M}\operatorname{res}_{\kappa(C_i)/\kappa(M)}(\gamma_i-\gamma_i)\\
            &=0.
        \end{align*}
        \end{itemize}
\end{proof}
\begin{lem}\label{cod2i}
    Let $X$ be as in (\ref{fibration}), and let $\beta_\zeta$ be as in (\ref{betadef}). Let $v$ be a valuation on $k(X)$ centered at a codimension 2 point $M\in S$ which is a closed point on only one $C_i$. Then $\partial^2_v(\beta_\zeta)=0$.
\end{lem}
\begin{proof}
    If $m_i=0$, then for all curves $E$ through $M$, 
    $$
    \partial^2_E(\alpha_\zeta)=0.
    $$ 
    By theorem \ref{pur}  and the cohomological purity for discrete valuation rings  (see \cite[Section 3.3]{CT92}), $\alpha_\zeta\in H^2_{\acute{e}t}(\mathcal{O}_{S,M},\Z/2\Z)$ and the claim follows.

    Similarly, if $m_i=1$, then by the definition of $\alpha_{\zeta}$ we have that  for all curves $E$ through $M$, 
    $$
    \partial^2_E(\alpha-\alpha_\zeta)=0,
    $$ 
    hence $\alpha-\alpha_\zeta\in H^2_{\acute{e}t}(\mathcal{O}_{S,M},\Z/2\Z)$. The claim follows since the image of $\alpha$ in $k(X)$ is zero.

\end{proof}
\begin{lem}\label{cod2ij}
    Let $X$ be as in (\ref{fibration}), and let $\beta_\zeta$ be as in (\ref{betadef}). Let $v$ be a valuation on $k(X)$ centered at a codimension 2 point $M\in S$ which is the intersection of $C_i$ and $C_j$. If either one of the following is true:
    \begin{itemize}
        \item[1)] $m_i=m_j$;
        \item[2)] $m_i=1\neq m_j$ and $\overline{\gamma_i}=1$ in $\kappa(M)^\times/\kappa(M)^{\times2}$,
    \end{itemize} 
    then $\partial^2_v(\beta_\zeta)=0$.
\end{lem}
\begin{proof}
    If $m_i=m_j=0$, then by Theorem \ref{pur}, $\alpha_\zeta\in H^2_{\acute{e}t}(\mathcal{O}_{S,M},\Z/2\Z)$, and the claim follows. 
    
    If $m_i=m_j=1$, then $\alpha-\alpha_\zeta\in
    H^2_{\acute{e}t}(\mathcal{O}_{S,M},\Z/2\Z)$ and the  claim follows since the image of $\alpha$ in $k(X)$ is zero.

    If $m_i=1\neq m_j$, then by local reciprocity we must have 
    $$
    \partial^1_M(\gamma_i)=\partial^1_M(\gamma_j)=0.
    $$ 
    By Theorem \ref{S2.1}, we can write
    $$
    \alpha_\zeta=(u,\pi_i)+\alpha_\zeta^\prime,
    $$ 
    where $u\in \mathcal{O}^\times_{S,M}$, 
      $\alpha_\zeta^\prime$ is in $ H^2_{\acute{e}t}(\mathcal{O}_{S,M},\Z/2\Z)$, 
      (so that the image of $\alpha_\zeta^\prime$ in $H^2(k(X),\Z/2\Z)$ is unramified at $v$), and  $\gamma_i=\partial_{C_i}(\alpha_{\zeta})=\bar u$.
    Since $\overline{\gamma_i}=1$, we deduce that the image of $u$ in $\kappa(M)$ is a square, hence, the image of $u$ in the completion $\widehat{\mathcal{O}}_{S,M}$ is a square. We deduce that  if $K_M$ is the field of fractions of $\widehat{\mathcal{O}}_{S,M}$ then the image of $(u,\pi_i)$ in $H^2(K_M,\Z/2\Z)$ is trivial. Since the residue $\partial^2_v$ factors though the completion,  we deduce that $\partial^2_v(\beta_\zeta)=\partial^2_v(u,\pi_i)=0$ using (\ref{center}).

\end{proof}

\subsection{The formula}

We first establish the converse of lemma \ref{cod2ij} in our setup. In particular, we show that at a codimension $2$ point $M\in S$ which is the intersection of two ramification curves, if $m_i=1\neq m_j$, then $\overline{\gamma}_i=1$ in $\kappa(M)^\times/\kappa(M)^{\times2}$ is a necessary and sufficient condition for $\beta_\zeta$ to be unramified at valuations centered at $M$. 
\begin{lem}\label{Scor}
    Let $\alpha_\zeta\in H^2(F,\Z/2\Z)$ have ramification divisor $\sum m_iC_i$. Fix a closed point $M\in C_i\cap C_j$ with $m_i=1$ and $m_j=0$. 
    Let $\overline{\gamma_i}$ and $\overline{\gamma_j}$ be the reductions of $\gamma_i$ and $\gamma_j$ in $H^1(\kappa(M),\Z/2\Z)$.  
    Then 
    \begin{itemize}
        \item [1.]$\alpha=\alpha'+(a_i,\pi_i)+(a_j,\pi_j)$, where $a_i,a_j\in \mathcal{O}^{\times}_{S,M}$ are lifts of $\gamma_i$ and $\gamma_j$, and $\alpha'\in H^2_{\acute{e}t}(\mathcal{O}_{S,M},\Z/2\Z)$.
        \item[2.] $\overline{\gamma_i}=\overline{\gamma_j}$.
    \end{itemize}
    
\end{lem}
\begin{proof}
    By Theorem \ref{S2.1}, $\alpha$ has two possible local forms at the intersection $M$:
    \begin{itemize}
        \item [i)] $\alpha=\alpha'+(a_i,\pi_i)+(a_j,\pi_j)$
        \item [ii)] $\alpha=\alpha'+(u\pi_j,v\pi_i)$
    \end{itemize}
    where $u,v\in \mathcal{O}_{S,M}^\times$ and $\alpha'\in H^2_{\acute{e}t}(\mathcal{O}_{S,M},\Z/2\Z)$. We show that form ii) is impossible under our assumption. Indeed, if $\alpha=\alpha'+(u\pi_j,v\pi_i)$, then $\gamma_i=\partial^2_{C_i}(\alpha)=u\pi_j$, and $\partial^1_M(\partial^2_{C_i}(\alpha))=\partial^1_M(u\pi_j)=1$. However, this breaks the local reciprocity of $\alpha_\zeta$ given by Gersten's complex:
    $$
    m_i\partial^1_M(\partial^2_{C_i}(\alpha))+m_j\partial^1_M(\partial^2_{C_j}(\alpha))=\partial^1_M(\gamma_i)=1\neq0.
    $$
    Hence $\alpha$ has to be of form i).

    Now assume $\overline{\gamma_i}$ and $\overline{\gamma_j}$ are not equal for the sake of contradiction. By Theorem \ref{S2.5}, $\alpha$ has index $>2$. However, $\alpha$ is a symbol corresponding to the generic conic, so it has index $\leq2$. We have reached a contradiction.
\end{proof}

\begin{prop}\label{trivialgamma}
   Let $\alpha_\zeta\in H^2(F,\Z/2\Z)$ have ramification divisor $\sum m_iC_i$. Fix a closed point $M\in C_i\cap C_j$ with $m_i=1$ and $m_j=0$. Denote by $\beta_\zeta$ the image of $\alpha_\zeta$ under the map $H^2(F,\Z/2\Z)\to H^2(k(X),\Z/2\Z)$. Then the following are equivalent:
\begin{itemize}
    \item [1.] $\beta_\zeta$ is unramified at any valuation centered at $M$;
    \item [2.] The reduction of $\gamma_i$ at $M$ is trivial in $H^1(\kappa(M),\Z/2\Z)$.
\end{itemize} 
\end{prop}
\begin{proof}  
The direction $2 \Rightarrow 1$ is proved by Lemma \ref{cod2ij}. We now prove  $1 \Rightarrow 2$.
Denote by $\pi_i$ and $\pi_j$ local parameters of $C_i$ and $C_j$ at $M$ respectively.
By Theorem \ref{S2.1}, we can write
$$\alpha=\alpha'+(a_i,\pi_i)+(a_j,\pi_j),$$ where $a_i,a_j\in \mathcal{O}^{\times}_{S,M}$ are lifts of $\gamma_i$ and $\gamma_j$ and by  Lemma \ref{Scor}, $\overline{\gamma_i}=\overline{\gamma_j}$ for  the reductions of $\gamma_i$ and $\gamma_j$ at $M$.

We then write
$$
\alpha_\zeta=(a_i,\pi_i)+\alpha_{\zeta}^\prime
$$
where  $\alpha_\zeta'$ is unramified at all curves through $M$, hence by Theorem \ref{pur}, $\alpha_\zeta^\prime\in H^2_{\acute{e}t}(\mathcal{O}_{S,M},\Z/2\Z)$.

    Denote $b:S'\to S$ the blow up of $S$ at $M$, $X'\coloneqq X\times_S S'$. We introduce a new local coordinate chart with substitution $\pi_i=\pi_i$ and $\pi_j=\pi_i{\pi_j^\prime}$, so $E\coloneqq \pi_i=0$ is the exceptional divisor and $\pi_j^\prime=0$ is the strict transform of $C_2$. Under this blow up,
    \begin{align*}
        \alpha&=(a_i,\pi_i)+(a_j,\pi_j)+ \alpha'\\
        &=(a_i,\pi_i)+(a_j,\pi_i\pi_j^\prime)+ \alpha'\\
        &=(a_ia_j,\pi_i)+(a_j,\pi_j^\prime)+\alpha'
    \end{align*}
    and
    \begin{align*}
        \partial^2_E(\alpha)=\partial^2_{E}(a_ia_j,\pi_i)+\partial^2_E(a_j,\pi_j^\prime)=\overline{a_i}+\overline{a_j},
    \end{align*}
    where $\overline{a_i},\overline{a_j}$ are the classes of $a_i,a_j$ in $H^1(\kappa(E),\Z/2\Z)$. Since $a_i,a_j$ are lifts of $\gamma_i,\gamma_j$ in $\mathcal{O}_{S,M}$, we have
    $$
   \overline{a_i}= \operatorname{res}_{\kappa(M)/\kappa(E)}(\overline{\gamma_i})\mbox{ and }\overline{a_j}= \operatorname{res}_{\kappa(M)/\kappa(E)}(\overline{\gamma_j}),
    $$
    where $\operatorname{res}_{\kappa(M)/\kappa(E)}$ is the restriction map from $H^1(\kappa(M),\Z/2\Z)$ to $H^1(\kappa(E),\Z/2\Z)$ induced by field extension.
    We then have
    $$
    \partial^2_E(\alpha)=\operatorname{res}_{\kappa(M)/\kappa(E)}(\overline{\gamma_i})+\operatorname{res}_{\kappa(M)/\kappa(E)}(\overline{\gamma_j})=0
    $$
    because $\overline{\gamma_i}=\overline{\gamma_j}$. We conclude $\alpha$ is unramified along $E$.

    Let $\widehat{\mathcal{O}}_{S',E}$ be the completion of the local ring $\mathcal{O}_{S',E}$ and $F_E$ its field of fractions. Let $\alpha_{F_E}$ be the base change of $\alpha$ to $F_E$. Note that $\alpha_{F_E}$ is still unramified over $\widehat{\mathcal{O}}_{S',E}$. Hence the  quaternion algebra $A$ over $F_E$ extends to an Azumaya algebra $\Lambda$ over $\widehat{\mathcal{O}}_{S',E}$ (\cite[Proposition 1.1 and 1.2]{F97}). This means we have a Severi-Brauer scheme $\mathcal{X}$ over $\widehat{\mathcal{O}}_{S',E}$ with generic fiber $\mathcal{X}_{F_E}\coloneqq X'_F\times_{F}F_E$ the base change of the generic fiber of $X'$. The special fiber $\mathcal{X}_{\eta_E}$ over the closed point of  $\widehat{\mathcal{O}}_{S',E}$ is hence a Severi-Brauer variety over $\kappa(E)$ of dimension $1$, which in particular means it is smooth and irreducible.

    Let $F_E(\mathcal{X})$ be the function field of the local Severi-Brauer scheme $\mathcal{X}$ and $v_{\mathcal{X}_{\eta_{E}}}$ be the valuation on $F_E(\mathcal{X})$ of the special fiber $\mathcal{X}_{\eta_E}$. There is a natural inclusion $k(X)\hookrightarrow F_E(\mathcal{X})$. By assumption, $\beta_\zeta$ is unramified at valuations centered at $M$, so its image in $H^2(F_E(\mathcal{X}),\Z/2\Z)$ must have trivial residue under the residue map  $\partial^2_{\mathcal{X}_{\eta_{E}}}$. 

    Let $v_E$ be the valuation on $F_E$ centered at $E$. Consider the following commutative diagram:
\[\begin{tikzcd}
	{{H^2(k(X),\Z/2\Z)}} && {H^2(F_E(\mathcal{X}),\Z/2\Z)} && {H^1(\kappa(E)(\mathcal{X}_{\eta_E}),\Z/2\Z)} \\
	\\
	{{H^2(F,\Z/2\Z)}} && {H^2(F_E,\Z/2\Z)} && {H^1(\kappa(E),\Z/2\Z)}
	\arrow["{\operatorname{res}_{k(X)/F_E(\mathcal{X})}}",  from=1-1, to=1-3]
	\arrow["{\partial^2_{\mathcal{X}_{\eta_{E}}}}", from=1-3, to=1-5]
	\arrow["{\pi^*}", from=3-1, to=1-1]
	\arrow["{\operatorname{res}_{F/F_E}}"', from=3-1, to=3-3]
	\arrow["{\operatorname{res}_{F_E/F_E(\mathcal{X})}}", from=3-3, to=1-3]
	\arrow["{\partial^2_E}"', from=3-3, to=3-5]
	\arrow["{\phi^*}"', hook',from=3-5, to=1-5]
\end{tikzcd}\]
    where $\phi^*$ is injective because the special fiber is smooth and irreducible. We have that
    \begin{align*}
        \partial^2_{E}(\operatorname{res}_{F/F_E}(\alpha_\zeta))&=\partial^2_E\Big( (a_i, \pi_i)_{F_E} + \operatorname{res}_{F/F_E}(\alpha^\prime) \Big)\\
      &=\partial^2_E\Big( (a_i, \pi_i)_{F_E} \Big)\\
      &=\gamma_i(M)\in H^1(\kappa(E),\Z/2\Z).     \end{align*}
    Since we have established that $\partial^2_{\mathcal{X}_{\eta_{E}}}(\operatorname{res}_{k(X)/F_E(\mathcal{X})}(\beta_\zeta))=0$, we have to have $\phi^*(\gamma_i(M))=0$ and hence $\gamma_i(M)=0$ by injectivity of $\phi^*$. Since $\gamma_i(M)$ is the image of $\overline{\gamma_i}$ by the injective map $H^1(\kappa(M),\Z/2\Z)\to H^1(\kappa(E),\Z/2\Z)$, we finish the proof.

\end{proof}

Putting this together we can provide a formula for $H^2_{nr}(k(X)/k,\Z/2\Z)$:
\begin{thm}\label{mainformula}
Let $k$ be a field of characteristic $char(k)\neq2$. Let $S$ be a smooth, projective, rational surface over $k$ and let $F$ be the field of functions of $S$. Let $X$ be a threefold equipped with a conic bundle structure $\pi: X\to S$ and let $\alpha\in Br(F)[2]$ be the class corresponding to the quaternion algebra $A$ associated to the conic given by the generic fiber $Q$ of $\pi$. Assume that $\alpha$ is nonzero and that the ramification curve 

$$C\coloneqq \left\{x\in S^{(1)}\mid \partial^2_x(\alpha)\neq0\right\}$$
has smooth irreducible components intersecting transverselly. Let $C=\cup_1^nC_i$ be the irreducible decomposition of $C$, and let 
$$\{(\gamma_i)\}_i\in \oplus_1^n H^1(\kappa(C_i),\Z/2\Z)$$ be the associated family of residues of $\alpha$. For $\zeta=\{(m_i)\}_i\in (\Z/2\Z)^n$, denote

    $$\gamma_{\zeta}=(\gamma_i^{m_i})\in\oplus_{x\in{S}^{(1)}}H^1(\kappa(x),\Z/2\Z).$$ 
    
    Consider the subgroup $H\subset(\Z/2\Z)^n$ of $n$-tuples $(m_i)$ such that 
    \begin{itemize}
    
    \item [(1)] for every $\zeta\in H$, $\gamma_{\zeta}$ lifts\footnote{ $\;$ i.e. there exists a $\alpha_\zeta\in H^2(F,\Z/2\Z)$ such that $\phi_2(\alpha_{\zeta})=\gamma_{\zeta}$ as in definition (\ref{deflifts})} to $H^2$. In particular, if $i\neq j$ and there is a point $P\in C_i\cap C_j$ such that $\partial^1_P(\gamma_i)=\partial^1_P(\gamma_j)\neq0$ then $m_i=m_j$;
        
        \item [(2)] if there is a point $P\in C_i\cap C_j$ such that $\partial^1_P(\gamma_i)=\partial^1_P(\gamma_j)=0$ but $m_i=1\neq m_j$, then $\overline{\gamma_i}=1$ in $\kappa(P)^\times/\kappa(P)^{\times2}.$
    \end{itemize}
    Then $H^2_{nr}(k(X)/k,\Z/2\Z)/H^2(k,\Z/2\Z)=H/(1,\ldots,1)$.
\end{thm}
\begin{proof}

    By Theorem \ref{ovvker}, the kernel of the restriction map $$\operatorname{res}_{F/k(X)}:H^s(F,\Z/2\Z)\rightarrow H^s(k(X),\Z/2\Z)$$ is $\Z/2\Z$ generated by $\alpha$. 
    
    Let $\beta_{0}\in H_{nr}^2(k(X)/k,\Z/2\Z)$. By \cite{Ara75}, $\beta_{0}$ is the image of an element $\alpha_0\in H^2(F,\Z/2\Z).$ We now analyze the residues of $\alpha_0$ at points  $x \in S^{(1)}$. Let $v$ be the discrete valuation on $F$ corresponding to $x$. Denote by $F_x$ the completion of $F$ at $x$. 
    
    If $x\neq C_i$ for $C_i\subset C$, then we proceed with a similar argument as in the proof of Proposition \ref{trivialgamma}. Indeed, $\alpha$ is unramified at $x$, and the conic $Q\times_F F_x$ has good reduction over $\widehat{\mathcal{O}}_{S,x}$. The special fiber is a Severi-Brauer variety over $\kappa(x)$, and the generic point of this special fiber yields a discrete valuation $w_x$ on $F_x(Q)$. Restricting $w_x$ to the function field $k(X)$ defines a valuation $w$ on $k(X)$ which extends $v$. Since residue fields are invariant under completion, and the special fiber is smooth and irreducible, the map $H^1(\kappa(x), \mathbb{Z}/2\mathbb{Z}) \to H^1(\kappa(w), \mathbb{Z}/2\mathbb{Z})$ is injective. Since $\beta_0 = \pi^*(\alpha_0)$ is unramified on $X$, the commutative diagram

\[\begin{tikzcd}
	{{H^2(k(X), \Z/2\Z)}} && {H^1(\kappa(w), \Z/2\Z)} \\
	\\
	{H^2(F, \Z/2\Z)} && {H^1(\kappa(x), \Z/2\Z)}
	\arrow["{\partial^2_w}", from=1-1, to=1-3]
	\arrow["{\pi^*}", from=3-1, to=1-1]
	\arrow["{\partial^2_x}"', from=3-1, to=3-3]
	\arrow["{\pi_x^*}"', from=3-3, to=1-3]
\end{tikzcd}\]
shows $\partial^2_x(\alpha_0)=0$.

    If $x=C_i$ for some $C_i\subset C$, then by \cite[Theorem 1.4]{A82}, there is a model of $X$ over $\widehat{\mathcal{O}}_{S,x}$ such that the special geometric fiber contains $2$ components of multiplicity $1$. These two components are both rational and are conjugate over $\kappa(x)$. Thus there exists a valuation $w$ on $k(X)$ which extends $v$ such that $\kappa(w)$ is a purely transcendental extension of $L$, where $L$ is an extension of $\kappa(x)$ of degree $2$. Since $\kappa(w)$ is a purely transcendental extension of $L$, the map $H^1(L,\Z/2\Z)\to H^1(\kappa(w),\Z/2\Z)$ is injective, and we have the commutative diagram 
\[\begin{tikzcd}
	{{H^2(k(X), \Z/2\Z)}} && {H^1(\kappa(w), \Z/2\Z)} \\
	&& {H^1(L, \Z/2\Z)} \\
	{H^2(F, \Z/2\Z)} && {H^1(\kappa(x), \Z/2\Z)}
	\arrow["{\partial^2_w}", from=1-1, to=1-3]
	\arrow["\operatorname{res}_{L/\kappa(w)}"', hook', from=2-3, to=1-3]
	\arrow["{\pi^*}", from=3-1, to=1-1]
	\arrow["{\partial^2_x}"', from=3-1, to=3-3]
	\arrow["\operatorname{res}_{\kappa(x)/L}"', from=3-3, to=2-3]
\end{tikzcd}.\]
In particular, $$\operatorname{Ker}[H^1(\kappa(x),\Z/2\Z)\to H^1(\kappa(w),\Z/2\Z)]=\operatorname{Ker}\operatorname{res}_{\kappa(x)/L}.$$

    Because $\alpha$ is split by $k(X)$, by the commutative diagram, $\partial_x^2(\alpha)=\gamma_i$ is in the kernel of $\operatorname{res}_{\kappa(x)/L}$. By assumption $\partial_x^2(\alpha)=\gamma_i$ is a nontrivial element in $H^1(\kappa(x),\Z/2\Z)$, so it corresponds to a cyclic Galois extension of $\kappa(x)$ of degree $2$. This extension coincides with $L$ because $\partial_x^2(\alpha)\in\operatorname{Ker}\operatorname{res}_{\kappa(x)/L}$ and the degree $[L:\kappa(x)]=2$. Hence the kernel of $\operatorname{res}_{\kappa(x)/L}$ is generated by $\gamma_i$. Since $\beta_0 = \pi^*(\alpha_0)$ is unramified on $X$, the commutative diagram shows that $\partial^2_x(\alpha_0)\in \operatorname{Ker}\operatorname{res}_{\kappa(x)/L}$. Therefore $\partial^2_x(\alpha_0)=m_i\gamma_i$ for some $m_i\in\{0,1\}$.

    We check this $n$-tuple $(m_i)$ satisfies the three conditions. Indeed, condition $(1)$ is trivially satisfied and condition $(2)$ is satisfied by Proposition \ref{trivialgamma}.

Let $$\Phi: H^2_{nr}(k(X)/k,\Z/2\Z)\rightarrow H/(1,\ldots,1)$$ be the homomorphism defined by residue data: as above, for an unramified class $\beta_0\in H^2_{nr}(k(X)/k,\Z/2\Z)$, its image in $H^2(k(X),\Z/2\Z)$ is the image of an element $\alpha_0\in H^2(F,\Z/2\Z)$ such that
    $$
    \partial^2_{C_i}(\alpha_0)=m_i\gamma_i,\mbox{ with }m_i\in\{0,1\},\mbox{ and }\partial^2_x(a_0)=0\mbox{ for } x\in S^{(1)}\setminus C
    $$
    and the $n$-tuple $\{(m_i)\}_i$ satisfies conditions $(1)$ and $(2)$ in the statement of the theorem. Define $$\Phi(\beta_0)=\{(m_i)\}_i.$$ This map is independent of the choice of $\alpha_0$ and hence well defined: if $\alpha_0$ and $\alpha_0^\#$ are both elements in $H^2(F,\Z/2\Z)$ mapping to the image of $\beta_0$, then $\alpha_0-\alpha_0^\#\in\operatorname{Ker}\operatorname{res}_{F/k(X)}=\langle\alpha\rangle$. Hence $\Phi(\alpha_0-\alpha_0^\#)\in\langle(1,\ldots,1)\rangle$.

    We first show that $\operatorname{Ker} \Phi=H^2(k,\Z/2\Z)$. The direction $H^2(k,\Z/2\Z)\subset \operatorname{Ker} \Phi$ is clear, so we prove the other direction. Assume $\beta_0\in \operatorname{Ker} \Phi$. Since two elements in $H^2(F,\Z/2\Z)$ map to the same image in $H^2(k(X),\Z/2\Z)$ if and only if they differ by $\alpha$, we may pick $\alpha_0$ such that $\partial^2_{C_i}(\alpha_0)=0$. Hence $\alpha_0\in H^2_{nr}(k(S)/k,\Z/2\Z)$. Since $S$ is rational, $H^2_{nr}(k(S)/k,\Z/2\Z)=H^2(k,\Z/2\Z).$
    
    Next, we show surjectivity. For every $\zeta\in H$, $\gamma_{\zeta}$ lifts to some $\alpha_{\zeta}$ in $H^2$, and we let $\beta_\zeta$ be the image of $\alpha_\zeta$ under the map $H^2(F,\Z/2\Z)\to H^2(k(X),\Z/2\Z)$.
    We show that $\beta_\zeta\in H^2_{nr}(k(X)/k,\Z/2\Z)$.

    Let $v$ be a valutation on $k(X)$ centered at $M\in S$. 

    If $M$ is the generic point, then $\beta_\zeta$ is in $H^2_{\acute{e}t}(\mathcal{O}_{S,M})$, and $\partial^2_v(\beta_\zeta)=0$.

    If $M$ is of codimension $\geq 1$, this is handled by Lemma \ref{cod1}, \ref{cod2i}, and \ref{cod2ij}.

\end{proof}

\nocite{*}
\printbibliography
\end{document}